\documentclass[12pt]{article}

\usepackage{amsmath}
\usepackage{amsfonts}
\usepackage{amsthm}
\usepackage[english]{babel}

\begin{document}

\title{Fundamental solutions for the super Laplace\\ and Dirac operators and all their natural powers}

\author{H.\ De Bie\footnote{Corresponding author, E-mail: {\tt Hendrik.DeBie@UGent.be}} \and F.\ Sommen\footnote{E-mail: {\tt fs@cage.ugent.be}}
}

\date{\small{Clifford Research Group -- Department of Mathematical Analysis}\\
\small{Faculty of Engineering -- Ghent University\\ Galglaan 2\\ 9000 Gent\\
Belgium}}

\maketitle

\abstract{The fundamental solutions of the super Dirac and Laplace operators and their natural powers are determined within the framework of Clifford analysis.
}

\noindent
\textbf{MSC 2000 :}   30G35, 35A08, 58C50\\
\noindent
\textbf{Keywords :}   Clifford analysis, Dirac operator, Superspace, Fundamental solution

\maketitle

\newtheorem{theorem}{Theorem}
\newtheorem{lemma}{Lemma}
\newtheorem{definition}{Definition}
\newtheorem{corollary}{Corollary}
\newtheorem{proposition}{Proposition}
\newtheorem{remark}{Remark}

\newcommand{\mR}{\mathbb{R}}
\newcommand{\mC}{\mathbb{C}}
\newcommand{\mN}{\mathbb{N}}
\newcommand{\mE}{\mathbb{E}}
\newcommand{\mZ}{\mathbb{Z}}

\newcommand{\cD}{\mathcal{D}}
\newcommand{\cM}{\mathcal{M}}
\newcommand{\cH}{\mathcal{H}}
\newcommand{\cE}{\mathcal{E}}
\newcommand{\cF}{\mathcal{F}}
\newcommand{\cC}{\mathcal{C}}
\newcommand{\cL}{\mathcal{L}}
\newcommand{\cP}{\mathcal{P}}

\newcommand{\ux}{\underline{x}}
\newcommand{\uxb}{\underline{x \grave{}}}
\newcommand{\uy}{\underline{y}}

\newcommand{\pj}{\partial_{x_j}}
\newcommand{\pjb}{\partial_{{x \grave{}}_{j}}}
\newcommand{\pkb}{\partial_{{x \grave{}}_{k}}}

\newcommand{\pI}{\partial_{x_i}}
\newcommand{\pIb}{\partial_{{x \grave{}}_{i}}}
\newcommand{\pIcont}{\partial_{x_i} \rfloor}
\newcommand{\pIbcont}{\partial_{{x \grave{}}_{i}} \rfloor}

\newcommand{\pk}{\partial_{x_k}}
\newcommand{\px}{\partial_x}
\newcommand{\py}{\partial_y}
\newcommand{\upx}{\partial_{\underline{x}}}
\newcommand{\upxb}{\partial_{\underline{{x \grave{}}} }}
\newcommand{\upxcont}{\partial_{\underline{x}}\rfloor}
\newcommand{\upxbcont}{\partial_{\underline{{x \grave{}}}}\rfloor}

\newcommand{\pxcont}{\partial_x \rfloor}
\newcommand{\pjcont}{\partial_{x_j} \rfloor}
\newcommand{\pjbcont}{\partial_{{x \grave{}}_{j}} \rfloor}
\newcommand{\pkcont}{\partial_{x_k} \rfloor}
\newcommand{\pkbcont}{\partial_{{x \grave{}}_{k}} \rfloor}

\newcommand{\fermdirac}{{e \grave{}}_{2j} \partial_{{x \grave{}}_{2j-1}} - {e \grave{}}_{2j-1} \partial_{{x \grave{}}_{2j}}}
\newcommand{\fermvec}{{e \grave{}}_{2i-1} {x \grave{}}_{2i-1} + {e \grave{}}_{2i} {x \grave{}}_{2i}}

\section{Introduction}

In a previous set of papers (see \cite{DBS1,DBS4,DBS2,DBS5}) we have developed the basic framework for Clifford analysis in superspace. Clifford analysis in standard Euclidean space is a function theory of the so-called Dirac operator and as such a generalization of the theory of holomorphic functions of one complex variable. Basic references are \cite{MR697564,MR1169463,MR1130821}. A superspace on the other hand is a generalization of the classical concept of space, where not only commuting variables are considered, but also a set of anti-commuting ones. These superspaces find their use in theoretical physics. We refer the reader to e.g. \cite{MR0208930,MR732126,MR565567,MR0580292,MR1175751}.

We have first established the basic algebraic framework necessary for developing a theory of Clifford analysis in superspace (see \cite{DBS1} and \cite{DBS4}). We have constructed the fundamental differential operators such as the Dirac and Laplace operators, the Euler and Gamma operators, etc. Next we have constructed a theory of spherical monogenics in superspace (see \cite{DBS2}), which was then used to introduce an integration over the supersphere and over superspace (see \cite{DBS5}), inspired by similarities with the classical theory of harmonic analysis. Moreover it turned out that this integration is equivalent with the one introduced in the work of Berezin (see \cite{MR0208930,MR732126}), although our approach offers better insight. Indeed, the definition given by Berezin is not motivated by any connection with classical types of integration, whereas our approach connects the integral over superspace with the well-known theory of integration in Euclidean space (see the discussion in \cite{DBS5}).

All previous work is situated on the level of polynomial functions. A next step is to consider larger algebras and to start the study of the function-theoretical properties of our differential operators. Thus is the aim of the present paper. More precisely we will determine the fundamental solutions for all natural powers of the super Dirac operator.

The paper is organized as follows. We start with a short introduction to Clifford analysis on superspace. Then we recapitulate the results on polyharmonic functions needed later on. In the next section we derive the fundamental solutions for the super Laplace and Dirac operators and compare them with an ad hoc approach, inspired by the theory of radial algebra. In the following section we extend this technique to construct fundamental solutions for all natural powers. Finally we discuss how the present technique is valid for a still larger class of differential operators.

\section{Clifford analysis in superspace}

We first consider the real algebra $\cP = \mbox{Alg}(x_i, e_i; {x \grave{}}_j,{e \grave{}}_j)$, $i=1,\ldots,m$, $j=1,\ldots,2n$
generated by

\begin{itemize}
\item $m$ commuting variables $x_i$ and $m$ orthogonal Clifford generators $e_i$
\item $2n$ anti-commuting variables ${x \grave{}}_i$ and $2n$ symplectic Clifford generators ${e \grave{}}_i$
\end{itemize}
subject to the multiplication relations
\[ \left \{
\begin{array}{l} 
x_i x_j =  x_j x_i\\
{x \grave{}}_i {x \grave{}}_j =  - {x \grave{}}_j {x \grave{}}_i\\
x_i {x \grave{}}_j =  {x \grave{}}_j x_i\\
\end{array} \right .
\quad \mbox{and} \quad
\left \{ \begin{array}{l}
e_j e_k + e_k e_j = -2 \delta_{jk}\\
{e \grave{}}_{2j} {e \grave{}}_{2k} -{e \grave{}}_{2k} {e \grave{}}_{2j}=0\\
{e \grave{}}_{2j-1} {e \grave{}}_{2k-1} -{e \grave{}}_{2k-1} {e \grave{}}_{2j-1}=0\\
{e \grave{}}_{2j-1} {e \grave{}}_{2k} -{e \grave{}}_{2k} {e \grave{}}_{2j-1}=\delta_{jk}\\
e_j {e \grave{}}_{k} +{e \grave{}}_{k} e_j = 0\\
\end{array} \right .
\]
and where moreover all elements $e_i$, ${e \grave{}}_j$ commute with all elements $x_i$, ${x \grave{}}_j$.

\noindent
If we denote by $\Lambda_{2n}$ the Grassmann algebra generated by the anti-commuting variables ${x \grave{}}_j$ and by $\cC$ the algebra generated by all the Clifford numbers $e_i$, ${e \grave{}}_j$, then we clearly have that

\[
\cP = \mR[x_1,\ldots,x_m]\otimes \Lambda_{2n} \otimes \cC.
\]
The most important element of the algebra $\cP$ is the vector variable $x = \ux+\uxb$ with

\[
\begin{array}{lll}
\ux &=& \sum_{i=1}^m x_i e_i\\
&& \vspace{-2mm}\\
\uxb &=& \sum_{j=1}^{2n} {x \grave{}}_{j} {e \grave{}}_{j}.
\end{array}
\]
One calculates that

\[
x^2 = \uxb^2 +\ux^2 = \sum_{j=1}^n {x\grave{}}_{2j-1} {x\grave{}}_{2j}  -  \sum_{j=1}^m x_j^2.
\]
The super Dirac operator is defined as

\[
\px = \upxb-\upx = 2 \sum_{j=1}^{n} \left( {e \grave{}}_{2j} \partial_{{x\grave{}}_{2j-1}} - {e \grave{}}_{2j-1} \partial_{{x\grave{}}_{2j}}  \right)-\sum_{j=1}^m e_j \pj.
\]
Its square is the super Laplace operator:

\[
\Delta = \px^2 =4 \sum_{j=1}^n \partial_{{x \grave{}}_{2j-1}} \partial_{{x \grave{}}_{2j}} -\sum_{j=1}^{m} \pj^2.
\]
The bosonic part of this operator is $\Delta_b = -\sum_{j=1}^{m} \pj^2$, which is the classical Laplace operator. The fermionic part is $\Delta_f = 4 \sum_{j=1}^n \partial_{{x \grave{}}_{2j-1}} \partial_{{x \grave{}}_{2j}}$.
For the other important operators in super Clifford analysis and their commutation relations we refer the reader to \cite{DBS1}. If we let $\px$ act on $x$ we find that

\[
\px x = m-2n = M
\]
where $M$ is the so-called super-dimension. This numerical parameter gives a global characterization of our superspace and will be used in remark \ref{radialalgebraapproach}.
We will also need the following basic formulae (see \cite{DBS1})

\begin{eqnarray}
\label{evenexpr}
\px(x^{2s} ) &=& 2 s x^{2s-1}\\
\label{oddexpr}
\px(x^{2s+1} ) &=& (M + 2s) x^{2s}.
\end{eqnarray}
In the case where $m=1$, $n=0$ this reduces to the familiar formula $\frac{d}{dx} x^k = k x^{k-1}$.

\noindent
Now we can consider several generalizations of the algebra $\cP$. This leads to the introduction of the  function-spaces:

\[
\cF(\Omega)_{m|2n} = \cF(\Omega) \otimes \Lambda_{2n} \otimes \cC
\]
where $\cF(\Omega)$ stands for $\cD(\Omega)$, $\cC^{k}(\Omega)$, $L_{p}(\Omega)$, $L_{1}^{\mbox{\footnotesize loc}}(\Omega)$, $\ldots$ with $\Omega$ an open domain in $\mR^m$. Finally the space of harmonic functions, i.e. null-solutions of the super Laplace operator, will be denoted by $\cH(\Omega)_{m|2n} \subseteq \cC^{2}(\Omega)\otimes \Lambda_{2n}$. Similarly we denote by $\cM(\Omega)_{m|2n} \subseteq \cC^{2}(\Omega)_{m|2n}$ the space of monogenic functions, i.e. null-solutions of the super Dirac operator. We have that $\cM(\Omega)_{m|2n} \subseteq  \cH(\Omega)_{m|2n} \otimes \cC$.

Now we have the following theorem, which generalizes a classical result in harmonic analysis.

\begin{theorem}
Null-solutions of the super Laplace and the super Dirac operator are $\cC^{\infty}$-functions, i.e.
\begin{eqnarray*}
\cH(\Omega)_{m|2n}\otimes \cC &\subseteq& \cC^{\infty}(\Omega)_{m|2n}\\
\cM(\Omega)_{m|2n} &\subseteq& \cC^{\infty}(\Omega)_{m|2n}.
\end{eqnarray*}
\end{theorem}

\begin{proof}
It suffices to give the proof for harmonic functions, as monogenic functions are also harmonic. So we consider a function $f \in \cH(\Omega)_{m|2n}$. Such a function can be written as

\[
f = \sum_{(\alpha)} f_{(\alpha)} {x \grave{}}_1^{\alpha_1} \ldots {x \grave{}}_{2n}^{\alpha_{2n}}
\]
with $(\alpha)=(\alpha_{1},\ldots,\alpha_{2n})$, $\alpha_{i} \in \{0,1\}$ and $f_{(\alpha)} \in \cC^2(\Omega)$.

\noindent
Expressing that $\Delta f = \Delta_b f + \Delta_f f=0$ leads to a set of equations of the following type

\[
\Delta_b f_{(\alpha)} = \sum_{(\beta)} c_{(\beta)} f_{(\beta)}, \quad |(\beta)| = |(\alpha)|+2
\]
with $c_{(\beta)} \in \mR$. We conclude that for every $(\alpha)$ there exists a $k \in \mN$ ($k\leq n+1$) such that $\Delta_b^k f_{(\alpha)}=0$. Hence $f_{(\alpha)}$ is polyharmonic and thus an element of $\cC^{\infty}(\Omega)$.
\end{proof}

We end this section with a few words on integration in superspace. The proper integral to consider is the so-called Berezin integral $\int_B$ (see \cite{MR0208930,MR732126} and \cite{DBS5}) which has the following formal definition

\[
\int_B = \int_{\mR^m} dV(\ux) \; \partial_{{x \grave{}}_{2n}} \ldots \partial_{{x \grave{}}_{1}},
\]
with $d V( \ux)$ the Lebesgue measure in $\mR^m$.

\noindent
One can also define a super Dirac distribution as

\[
\delta(x) = \delta(\ux) {x \grave{}}_{1} \ldots {x \grave{}}_{2n} =\delta(\ux) \frac{\uxb^{2n}}{n!}
\]
with $\delta(\ux)$ the classical Dirac distribution in $\mR^m$.
We clearly have that

\begin{eqnarray*}
<\delta(x-y), f(x)> &=& \int_B \delta(\ux -\uy ) ({x \grave{}}_{1}-{y \grave{}}_{1}) \ldots ({x \grave{}}_{2n}-{y \grave{}}_{2n}) f(x)\\
&=& f(y)
\end{eqnarray*}
with $f \in \cD(\Omega)_{m|2n}$.

\section{Fundamental solutions in $\mR^m$}
\label{fundsolRm}

The fundamental solutions for the natural powers of the classical Laplace operator $\Delta_b$ are very well known, see e.g. \cite{MR745128}. 

\noindent
We denote by $\nu_{2l}^{m|0}$, $l=1, 2, \ldots$ a sequence of such fundamental solutions, satisfying
\begin{eqnarray*}
\Delta^j_b \nu_{2l}^{m|0} &=& \nu_{2l-2j}^{m|0}, \quad j<l\\
\Delta^l_b \nu_{2l}^{m|0} &=& \delta(\ux).
\end{eqnarray*}
Their explicit form depends both on the dimension $m$ and on $l$.
More specifically, in the case where $m$ is odd we have that
\begin{equation}
\nu_{2l}^{m|0} = \frac{r^{2l-m}}{\gamma_{l-1}},\qquad \gamma_l =  (-1)^{l+1} (2-m) 4^l l ! \frac{\Gamma(l+2-m/2)}{\Gamma(2-m/2)} \frac{2 \pi^{m/2}}{\Gamma(m/2)}
\label{fundsolRmexpl}
\end{equation}
with $r = \sqrt{-\ux^2}$.
The formulae for $m$ even are more complicated and can be found in \cite{MR745128}.

\noindent
Concerning the refinement to Clifford analysis, we clearly have that $\nu_{2l+1}^{m|0} = \upx \nu_{2l+2}^{m|0}$ is a fundamental solution of $\Delta^l_b \upx$.

\section{Fundamental solution of $\Delta$ and $\px$}
\label{fundsollapl}

From now on we restrict ourselves to the case where $m \neq 0$. The purely fermionic case will be discussed briefly in section \ref{fermcase}.

\noindent
Our aim is to construct a function $\rho$ such that in distributional sense

\[
\Delta \rho = \delta(x).
\]

\noindent
We propose the following form for the fundamental solution:

\[
\rho =  \sum_{k=0}^n a_k (\Delta_b^{n-k}\phi) \uxb^{2n-2k},
\]
with $\phi$ and $a_k \in \mR$ to be determined.

\noindent
Now let us calculate $\Delta \rho$

\begin{eqnarray*}
\Delta \rho &=& (\Delta_b+\Delta_f) \rho\\
&=&  \sum_{k=0}^n a_k  (\Delta_b^{n-k +1}\phi) \uxb^{2n-2k}\\
&& +  \sum_{k=1}^n a_k (2n-2k)(2n-2k-2-2n) (\Delta_b^{n-k}\phi) \uxb^{2n-2k-2}\\
&=& a_0 (\Delta_b^{n+1}\phi) \uxb^{2n} + \sum_{k=1}^n \left[ a_k - 2k(2n-2k+2)a_{k-1}  \right] (\Delta_b^{n-k +1}\phi) \uxb^{2n-2k}.\\
\end{eqnarray*}

\noindent
So $\rho$ is a fundamental solution if and only if

\[
a_0 (\Delta_b^{n+1}\phi) = \delta(\ux) \frac{1}{n!}
\]
and $a_k$ satisfies the recurrence relation

\[
a_k =4k(n-k+1) a_{k-1}.
\]

\noindent
The first equation leads to

\[
\phi = \nu_{2n+2}^{m|0},\quad a_0 = \frac{1}{n!}.
\]

\noindent
We then immediately find the following expression for the $a_k$

\[
a_k = \frac{ 4^k k!}{(n-k)!},\qquad k = 0, \ldots ,n.
\]

\noindent
Summarizing we obtain the following theorem

\begin{theorem}
The function $\nu_2^{m|2n}$ defined by

\[
\nu_2^{m|2n} = \sum_{k=0}^n \frac{4^k k!}{(n-k)!} \nu_{2k+2}^{m|0} \uxb^{2n-2k},
\]
with $\nu_{2k+2}^{m|0}$ as in section \ref{fundsolRm},
is a fundamental solution for the operator $\Delta$.
\label{fundsollaplace}
\end{theorem}

\begin{proof}
It is clear that $\nu_2^{m|2n} \in L_{1}^{\mbox{\footnotesize loc}}(\mR^m)_{m|2n}$. Moreover, we have that $\nu_2^{m|2n} \in \cH(\mR^m -  \{0\})_{m|2n}$ and that $\Delta \nu_2^{m|2n} = \delta(x)$ in distributional sense.
\end{proof}

\noindent
Now suppose that $m$ is odd, then the previous formula simplifies to
\begin{equation}
\nu_2^{m|2n} = \frac{\Gamma(m/2)}{2(2-m) \pi^{m/2}} \sum_{k=0}^n  \frac{ (-1)^{k+1}}{(n-k)!} \frac{\Gamma(2-m/2)}{\Gamma(k+2-m/2)} r^{2k+2-m} \uxb^{2n-2k},
\label{specialcaselaplace}
\end{equation}
where we have used formula (\ref{fundsolRmexpl}).

\noindent
As we have that $\Delta \nu_2^{m|2n} = \delta(x)$, a fundamental solution for the Dirac operator $\px$ is given by $\px \nu_2^{m|2n}$. This leads to the following 

\begin{theorem}
The function $\nu_1^{m|2n}$ defined by

\begin{eqnarray*}
\nu_1^{m|2n} &=& \sum_{k=0}^{n-1} 2 \frac{4^k k!}{(n-k-1)!} \nu_{2k+2}^{m|0} \uxb^{2n-2k-1} - \sum_{k=0}^n \frac{4^k k! }{(n-k)!} \nu_{2k+1}^{m|0} \uxb^{2n-2k}
\end{eqnarray*}
is a fundamental solution for the operator $\px$.
\end{theorem}

\begin{remark}
We could propose the following form

\[
g = \frac{1}{(x^2)^{\frac{M-2}{2}}}
\]
for the fundamental solution of $\Delta$, where we have replaced $m$ by the super-dimension $M$ in the classical expression. This technique is inspired by radial algebra (see \cite{MR1472163}), which gives a very general framework for constructing theories of Clifford analysis, based on the introduction of an abstract dimension parameter (in this case the super-dimension). This leads partially to the correct result (see also \cite{MR2032707}). Indeed, formally we can expand this as

\begin{eqnarray*}
g &=& \frac{1}{(x^2)^{\frac{M-2}{2}}}\\
&=& \frac{1}{(\ux^2 +\uxb^2)^{\frac{M-2}{2}}}\\
&=&\frac{1}{(\ux^2)^{\frac{M-2}{2}}} \left( 1 + \frac{\uxb^2}{\ux^2} \right)^{1-\frac{M}{2}}\\
&=& \sum_{k=0}^n \binom{1-\frac{M}{2}}{k}  \frac{\uxb^{2k}}{(\ux^2)^{\frac{M}{2} -1 +k}}.
\end{eqnarray*}
The coefficients in this development are proportional to the ones obtained in theorem \ref{fundsollaplace} (see also formula (\ref{specialcaselaplace})), so this  yields the correct result. This expansion is however only valid if $m$ is odd.
\label{radialalgebraapproach}
\end{remark}

The fundamental solution can of course be used to determine solutions of the inhomogeneous equation $\Delta f = \rho$. We have for example the following

\begin{proposition}
Let $\rho \in \cD(\Omega)_{m|2n}$, then a solution of $\Delta f = \rho$ is given by

\[
f(x)= \nu^{m|2n}_2 * \rho = \int_B \nu^{m|2n}_2(x-y) \rho(y).
\]
\end{proposition}

\section{Fundamental solution of $\Delta^k$ and $\Delta^k \px$}
\label{fundsoliterlapl}

A similar technique as in section \ref{fundsollapl} can be used for the polyharmonic case. First we expand $\Delta^k$ as

\[
\Delta^k = \sum_{j=0}^{k} \binom{k}{j} \Delta_b^{k-j} \Delta_f^j.
\]
This expansion is valid as $\Delta_b$ commutes with $\Delta_f$.

\noindent
Now we propose the following form for its fundamental solution:

\[
\rho =  \sum_{l=0}^n a_l \left( \Delta_b^{n-l} \phi \right) \uxb^{2n-2l}
\]
with $\phi$ and $a_l \in \mR$ still to be determined. We calculate that

\[
\Delta^k \rho = \sum_{l=0}^n a_l  \sum_{j=0}^{k} \binom{k}{j}  \left( \Delta_b^{n-l+k-j} \phi \right) \Delta_f^j \uxb^{2n-2l}.
\]
As we have that, using formulae (\ref{evenexpr}) and (\ref{oddexpr}) in the case where $m=0$, $M=-2n$, 

\[
\Delta_f^j \uxb^{2n-2l} = 4^j (-1)^j \frac{(n-l)!}{(n-l-j)!}\frac{(l+j)!}{l!}\uxb^{2n-2l-2j}, \quad j \leq n-l
\]
this yields

\[
\Delta^k \rho =\sum_{l=0}^n a_l  \sum_{j=0}^{k} \binom{k}{j}   4^j (-1)^j \frac{(n-l)!}{(n-l-j)!}\frac{(l+j)!}{l!} \left( \Delta_b^{n-l+k-j} \phi \right)\uxb^{2n-2l-2j}.
\]
Putting $\Delta^k \rho = \delta(x)$ leads to the following set of equations

\begin{eqnarray}
a_0 \Delta_b^{n+k} \phi &=& \delta(\ux) \frac{1}{n!}\\
\label{iteratedsequence}
\sum_{j=0}^{k} a_{l-j} \binom{k}{j}   4^j (-1)^j \frac{(n-l+j)!}{(n-l)!}\frac{l!}{(l-j)!}&=&0\\
a_{-1}=a_{-2}=a_{-3}=\ldots &=&0.
\end{eqnarray}
We immediately have that $a_0 =1/n!$ and that $\phi = \nu_{2n+2k}^{m|0}$.

\noindent
Equation (\ref{iteratedsequence}) can be simplified by the substitution

\[
a_l = 4^l \frac{l ! }{(n-l)!}  b_l
\]
to

\[
\sum_{j=0}^{k} b_{l-j} \binom{k}{j} (-1)^j =0,\qquad b_0=1
\]
which has the solution (see the subsequent lemma \ref{auxlemma})

\[
b_l = \binom{l+k-1}{l}.
\]
We conclude that 

\[
a_l = 4^l \frac{(l+k-1)! }{(n-l)! (k-1)!} , \quad l=0,\ldots,n.
\]

\noindent
We can summarize the previous results in the following theorem.
\begin{theorem}
The function $\nu_{2k}^{m|2n}$ defined by

\[
\nu_{2k}^{m|2n} = \sum_{l=0}^n 4^l \frac{(l+k-1) ! }{(n-l)! (k-1)!} \nu_{2l+2k}^{m|0} \uxb^{2n-2l},
\]
is a fundamental solution for the operator $\Delta^k$.
\label{fundsollaplaceiterated}
\end{theorem}

\noindent
In a similar vein we obtain the fundamental solution $\nu_{2k+1}^{m|2n}$ for the operator $\Delta^k \px = \px^{2k+1}$ by calculating $\px \nu_{2k+2}^{m|2n}$. This leads to

\begin{theorem}
The function $\nu_{2k+1}^{m|2n}$ defined by

\begin{eqnarray*}
\nu_{2k+1}^{m|2n} &=& \sum_{l=0}^{n-1} 2  \frac{4^l(l+k) ! }{(n-l-1)! k!} \nu_{2l+2k+2}^{m|0} \uxb^{2n-2l-1}\\
&& - \sum_{l=0}^n  \frac{ 4^l(l+k)!}{(n-l)! k!} \nu_{2l+2k+1}^{m|0} \uxb^{2n-2l},
\end{eqnarray*}
is a fundamental solution for the the operator $\Delta^k \px$.
\end{theorem}

\vspace{2mm}
\noindent
We still have to prove the technical lemma we used in the derivation of theorem \ref{fundsollaplaceiterated}.
\begin{lemma}
The sequence $(b_l)$, $l=0,1,\ldots$, recursively defined by

\[
\sum_{j=0}^{\min(k,l)} b_{l-j} \binom{k}{j}   (-1)^j =0, \quad b_0 = 1
\]
is given explicitly by

\[
b_l = \binom{l+k-1}{l}.
\]
\label{auxlemma}
\end{lemma}

\begin{proof}
Define the polynomial $R_l(x)$ by

\begin{eqnarray*}
R_l(x) &=& \sum_{j=0}^{\min(k,l)} (-1)^j \binom{k}{j} \frac{(x+k-j-1)!}{(x-j)!}\\
&=&  \sum_{j=0}^{\min(k,l)} (-1)^j \binom{k}{j} (x+k-1-j) \ldots (x+1-j).
\end{eqnarray*}
We then have to prove that $R_l(l) = 0$. 

\noindent
We distinguish between three cases.

\vspace{2mm}
\noindent
\textbf{1) $l\leq k-2$}

We claim that for all $t \leq k-2$

\[
R_t(x) = (-1)^t \binom{k-1}{t} (x+k-t-1)\ldots (x+1) (x-1) \ldots (x-t).
\]
This can be proven using induction. The case where $t=1$ is easily checked. So we suppose the formula holds for $t-1$, then we calculate

\begin{eqnarray*}
&&R_{t-1}(x) + (-1)^t \binom{k}{t} (x+k-1-t) \ldots (x+1-t)\\
&&= (-1)^t (x+k-t-1)\ldots (x+1) (x-1) \ldots (x-t+1)\\
&& \times \left( \binom{k}{t} x - \binom{k-1}{t-1} (x+k-t) \right)\\
&&= (-1)^t \binom{k-1}{t} (x+k-t-1)\ldots (x+1) (x-1) \ldots (x-t)\\
&&= R_t(x)
\end{eqnarray*}
which proves the hypothesis. Now clearly $R_l(l) = 0$.

\vspace{2mm}
\noindent
\textbf{2) $l=k-1$}

Using the previous results, it is shown that in this case $R_{k-1}(x)$ equals

\[
R_{k-1}(x) = - (-1)^k (x-1) \ldots (x+1-k)
\]
so $R_{k-1}(k-1) = 0$.

\vspace{2mm}
\noindent
\textbf{3) $l \geq k$}

Now we have that

\[
R_l(x) = R_k(x) = 0
\]
so this case is also proven.
\end{proof}

\section{A larger class of differential operators}

The technique used above can be extended to a larger class of differential operators. Suppose we consider an operator of the following form

\[
P = L(x,\px) + \Delta_f
\]
with $L(x,\px)$ an elliptic operator in $\mR^m$ and $\Delta_f$ the fermionic Laplace operator. Note that $L(x,\px)$ and $\Delta_f$ clearly commute. An interesting operator in this class is the super Helmholtz operator

\[
\Delta - \lambda^2, \qquad \lambda \in \mR
\]
with $L(x,\px) = \Delta_b - \lambda^2$.

\noindent
Denoting by $\mu_{2k}^{m|0}$ ($k=1,2,\ldots$) a set of fundamental solutions for the operators $L(x,\px)^k$ such that

\begin{eqnarray*}
L(x,\px)^j \mu_{2k}^{m|0} &=& \mu_{2k-2j}^{m|0}, \quad j<k\\
L(x,\px)^k \mu_{2k}^{m|0} &=& \delta(\ux)
\end{eqnarray*}
we can now use the same technique as in section \ref{fundsoliterlapl} to obtain a fundamental solution $\mu_{2k}^{m|2n}$ for the operator $P^k$. This leads to

\[
\mu_{2k}^{m|2n} = \sum_{l=0}^n 4^l \frac{(l+k-1)!}{(n-l)! (k-1)!} \mu_{2l+2k}^{m|0} \uxb^{2n-2l}.
\]

\section{The purely fermionic case}
\label{fermcase}

In this case there is no fundamental solution. Indeed, determining the fundamental solution of $\Delta_f$ requires  solving the algebraic equation

\[
\Delta_f \nu_{2}^{0|2n} = {x \grave{}}_{1} \ldots {x \grave{}}_{2n}
\]
which clearly has no solution, since there are no polynomials of degree higher than $2n$.

\section{Conclusions}

In this paper we have developed a technique to construct fundamental solutions for certain differential operators in superspace. In particular we have constructed the fundamental solutions of the natural powers of the super Dirac operator $\px^k$.

We envisage to use these fundamental solutions in a further development of the function theory of Clifford analysis in superspace. There is no doubt that they will play an important role in the generalization of e.g. the Cauchy and Hilbert transform to superspace.

\vspace{3mm}

\subsection*{Acknowledgement}
The first author is a research assistant supported by the Fund for Scientific Research Flanders (F.W.O.-Vlaanderen). He would like to thank Liesbet Van de Voorde for a discussion concerning section \ref{fundsoliterlapl}.


\begin{thebibliography}{10}

\bibitem{DBS1}
H. De~Bie, F. Sommen,
\newblock Correct rules for Clifford calculus on superspace.
\newblock {\em Accepted for publication in Adv. Appl. Clifford Algebras\/}.

\bibitem{DBS4}
H. De~Bie, F. Sommen,
\newblock A Clifford analysis approach to superspace.
\newblock {\em Accepted for publication in Ann. Physics\/}.

\bibitem{DBS2}
H. De~Bie, F. Sommen,
\newblock Fischer decompositions in superspace.
\newblock {\em Accepted for publication in the Proceedings of the 14th ICFIDCA\/}.

\bibitem{DBS5}
H. De~Bie, F. Sommen,
\newblock Spherical harmonics and integration in superspace.
\newblock {\em Accepted for publication in J. Phys. A\/}.

\bibitem{MR697564}
F. Brackx, R. Delanghe, F. Sommen,
\newblock {\em Clifford analysis}, {\em Research Notes in
  Mathematics}, vol.~76.
\newblock Pitman (Advanced Publishing Program), Boston, MA, 1982.

\bibitem{MR1169463}
R. Delanghe, F. Sommen, V. Sou{\v{c}}ek,
\newblock {\em Clifford algebra and spinor-valued functions}, vol.~53 of {\em
  Mathematics and its Applications}.
\newblock Kluwer Academic Publishers Group, Dordrecht, 1992.

\bibitem{MR1130821}
J.~E. Gilbert, M. A.~M. Murray,
\newblock {\em Clifford algebras and {D}irac operators in harmonic analysis},
  vol.~26 of {\em Cambridge Studies in Advanced Mathematics}.
\newblock Cambridge University Press, Cambridge, 1991.

\bibitem{MR0208930}
F.~A. Berezin,
\newblock {\em The method of second quantization}.
\newblock Translated from the Russian by Nobumichi Mugibayashi and Alan
  Jeffrey. Pure and Applied Physics, Vol. 24. Academic Press, New York, 1966.

\bibitem{MR732126}
F.~A. Berezin,
\newblock {\em Introduction to algebra and analysis with anticommuting
  variables}.
\newblock Moskov. Gos. Univ., Moscow, 1983.
\newblock With a preface by A. A. Kirillov.

\bibitem{MR565567}
D.~A. Le{\u\i}tes,
\newblock Introduction to the theory of supermanifolds.
\newblock {\em Uspekhi Mat. Nauk 35}, 1(211) (1980), 3--57, 255.

\bibitem{MR0580292}
B. Kostant,
\newblock Graded manifolds, graded {L}ie theory, and prequantization.
\newblock In {\em Differential geometrical methods in mathematical physics
  (Proc. Sympos., Univ. Bonn, Bonn, 1975)}. Springer, Berlin, 1977,
  pp.~177--306. Lecture Notes in Math., Vol. 570.

\bibitem{MR1175751}
C. Bartocci, U. Bruzzo, D. Hern{\'a}ndez~Ruip{\'e}rez,
\newblock {\em The geometry of supermanifolds},
{\em Mathematics and
  its Applications}, vol.~71. 
\newblock Kluwer Academic Publishers Group, Dordrecht, 1991.

\bibitem{MR745128}
N. Aronszajn, T.~M. Creese, L.~J. Lipkin,
\newblock {\em Polyharmonic functions.}
\newblock Oxford Mathematical Monographs. The Clarendon Press Oxford University
  Press, New York, 1983.

\bibitem{MR1472163}
F. Sommen,
\newblock An algebra of abstract vector variables.
\newblock {\em Portugal. Math. 54}, 3 (1997), 287--310.

\bibitem{MR2032707}
F. Sommen,
\newblock Clifford analysis on super-space. {II}.
\newblock In {\em Progress in analysis, Vol.\ I, II (Berlin, 2001)}. World Sci.
  Publishing, River Edge, NJ, 2003, pp.~383--405.

\end{thebibliography}
\end{document}